\documentclass{gtpart}
\usepackage{graphicx, amssymb}
\usepackage[y2k,lite]{amsrefs}
\usepackage{tikz}
\usetikzlibrary{shapes.geometric}
\usetikzlibrary{calc}
\newcommand{\dotnode}{ \node[fill,circle,inner sep=1pt]}
\newcommand{\tikzsetup}{
  \tikzstyle tikzdefault=[semithick,>=stealth];
  \tikzstyle overline=[preaction={draw,line
width=2mm,white,-},semithick,>=stealth];
  \tikzstyle cell=[fill=black!10,draw, very thin];
   \tikzstyle midarrow=[fill, dart, dart tail angle=112,
    dart tip angle=53,inner sep=0.55pt,sloped,allow upside down];
   \tikzstyle revmidarrow=[fill, dart, dart tail angle=112,
    dart tip angle=53,inner sep=0.55pt,sloped,
    allow upside down,rotate=180];
}
\newcommand{\OrientedCircle}[1][]{
  \begin{scope}[tikzdefault,#1]
  \draw[->] (-1,0) arc (-180:0:1);
  \draw[-] (1,0) arc (0:180:1);
  \end{scope}
}
\newcommand{\HalfDashedCircle}[1][]{
  \begin{scope}[tikzdefault,#1]
  \draw[densely dashed, -] (0,1) arc (90:270:1);
  \draw (0,-1) arc (-90:90:1);
  \end{scope}
}

\newcommand{\tzAnnulusWithRadius}{
\begin{tikzpicture}[baseline]
\tikzsetup
\draw[cell, semithick, even odd rule] (0,0) circle (0.3) 
               (0,0) circle (0.7);
\coordinate (a1) at (0,0.3);\dotnode at (a1){};
\coordinate  (a2) at (0,0.7);\dotnode at (a2){};
\draw[tikzdefault] (a1) --  (a2) ;
\end{tikzpicture}
}
\newcommand{\tzRectangle}{
\begin{tikzpicture}[baseline]
\tikzsetup

\coordinate (a1) at (0,0.4); \dotnode at (a1){};
      \node [above] at (a1) {$\scriptstyle a_1$};
\coordinate (a2) at (1.8,0.4); \dotnode at (a2){};
      \node [above] at (a2) {$\scriptstyle a_2$};

\coordinate (b1) at (0,-0.4); \dotnode at (b1){};
      \node [below] at (b1) {$\scriptstyle b_1$};
\coordinate (b2) at (1.8,-0.4); \dotnode at (b2){};
      \node [below] at (b2) {$\scriptstyle b_2$};
\draw[cell, semithick, even odd rule] (a1)--(a2)--(b2)--(b1)--cycle;
\end{tikzpicture}
}

\newcommand{\tzAnnulusWithRadiusLabeled}{
\begin{tikzpicture}[baseline]
\tikzsetup
\draw[cell, semithick, even odd rule] (0,0) circle (0.4) 
               (0,0) circle (0.8);
\coordinate (a1) at (0,0.4);\dotnode at (a1){}; 
      \node [below] at (a1) {$\scriptstyle b$};
\coordinate  (a2) at (0,0.8);\dotnode at (a2){};
      \node [above] at (a2) {$\scriptstyle a$};

\draw[tikzdefault] (a1) --  (a2) ;
\end{tikzpicture}
}
\newcommand{\tzAnnulusWithRadiusSubdivision}{
\begin{tikzpicture}[baseline]
\tikzsetup
\draw[cell, semithick, even odd rule] (0,0) circle (0.3) 
               (0,0) circle (0.7);
\coordinate (a1) at (0,0.3);\dotnode at (a1){};
\coordinate  (a2) at (0,0.7);\dotnode at (a2){};
\coordinate  (a3) at (0,-0.5);\dotnode at (a3){};

\draw[tikzdefault] (a1) --  (a2) ;
\draw[tikzdefault] (0,-0.1) circle (0.4);
\draw[tikzdefault] (0,0.1) circle (0.6);

\end{tikzpicture}
}
\newcommand{\tzCircle}{
\begin{tikzpicture}[baseline]
\tikzsetup
\draw[cell, semithick] (0,0) circle (0.7);
\coordinate  (a2) at (0,0.7){};\dotnode at (a2){};
\end{tikzpicture}
}
\newcommand{\tzCircleII}{
\begin{tikzpicture}[baseline]
\tikzsetup
\draw[cell, semithick] (0,0) circle (0.7);
\coordinate  (a2) at (0,0.7){};\dotnode at (a2){};
\coordinate  (a3) at (0,-0.7){};\dotnode at (a3){};
\end{tikzpicture}
}
\newcommand{\tzCircleIII}{
\begin{tikzpicture}[baseline]
\tikzsetup
\draw[cell, semithick] (0,0) circle (0.7);
\coordinate  (a2) at (0,0.7){};\dotnode at (a2){};
\coordinate  (a3) at (0,-0.7){};\dotnode at (a3){};
\draw[tikzdefault] (a2)--(a3);
\end{tikzpicture}
}
\newcommand{\tzCircleIV}{
\begin{tikzpicture}[baseline]
\tikzsetup
\draw[cell, semithick] (0,0) circle (0.7);
\coordinate  (a2) at (0,0.7){};\dotnode at (a2){};
\coordinate  (a3) at (0,-0.7){};\dotnode at (a3){};
\coordinate  (a4) at (0,0){};\dotnode at (a4){};
\draw[tikzdefault] (a2)--(a3);
\end{tikzpicture}
}
\newcommand{\tzCircleV}{
\begin{tikzpicture}[baseline]
\tikzsetup
\draw[cell, semithick] (0,0) circle (0.7);
\coordinate  (a2) at (0,0.7){};\dotnode at (a2){};
\coordinate  (a3) at (0,-0.7){};\dotnode at (a3){};
\coordinate  (a4) at (0,0){};\dotnode at (a4){};
\draw[tikzdefault] (a2)--(a4);
\end{tikzpicture}
}
\newcommand{\tzCircleWithRadius}{
\begin{tikzpicture}[baseline]
\tikzsetup
\draw[cell, semithick] (0,0) circle (0.7);
\coordinate  (a1) at (0,0){};\dotnode at (a1){};
\coordinate  (a2) at (0,0.7){};\dotnode at (a2){};
\draw[tikzdefault] (a1) --  (a2);
\end{tikzpicture}
}
\newcommand{\tzCircleWithRadiusOrient}{
\begin{tikzpicture}[baseline]
\tikzsetup
\fill[black!10] (0,0) circle (0.7);
\OrientedCircle[scale=0.7,semithick]
\coordinate  (a1) at (0,0){};\dotnode at (a1){};
\coordinate  (a2) at (0.7,0){};\dotnode at (a2){};
\draw[tikzdefault,->] (a1) --  (a2) node[pos=0.4,above] {$a$};
\node at (-40:0.85cm) {$b$};
\end{tikzpicture}
}
\newcommand{\tzTwoGlobe}{
\begin{tikzpicture}[baseline]
\tikzsetup
\coordinate  (a1) at (-1,0){};\dotnode at (a1){};
\coordinate  (a2) at (1,0){};\dotnode at (a2){};
\draw[fill=black!10,semithick, even odd rule] (a1) .. controls +(40:0.7cm) and +(140:0.7cm)..  (a2) 
 .. controls +(-140:0.7cm) and +(-40:0.7cm).. (a1);
\end{tikzpicture}
}
\newcommand{\tzZigZagComplex}{
\begin{tikzpicture}[baseline]
\tikzsetup
\coordinate (a1) at (-1,-0.5){};
\coordinate  (a2) at (1,-0.5){};
\coordinate  (a3) at (-1,0.8){};
\coordinate  (a4) at (0.2,0.8){};\coordinate  (a7) at (0.2,0.5){};\coordinate  (a8) at (0.2,-0.5){};
\coordinate  (a5) at (-0.2,0.5){};\coordinate  (a9) at (-0.2,-0.5){};
\coordinate  (a6) at (1,0.5){};
\filldraw[fill=black!5,very thin] (a4)--(a5)--(a7)--cycle;
\filldraw[cell,semithick] (a1) -- (a3) -- (a4) -- (a5) -- (a6) -- (a2) -- cycle;
\draw[thin, densely dashed](a7)--(a8);
\draw[very thin](a5)--(a9);
\dotnode at (a1){};\dotnode at (a2){};\dotnode at (a3){};\dotnode at (a6){};
\end{tikzpicture}
}
\newcommand{\tzPentagon}{
\begin{tikzpicture}[baseline]
\tikzsetup
\coordinate (a0) at (0:0.7cm){};
\coordinate (a1) at (72:0.7cm){};
\coordinate (a2) at (144:0.7cm){};
\coordinate (a3) at (216:0.7cm){};
\coordinate (a4) at (288:0.7cm){};
\draw[cell,semithick] (a0) --  (a1) -- (a2) -- (a3) -- (a4) -- cycle;
\dotnode at (a0){};\dotnode at (a1){};\dotnode at (a2){};\dotnode at (a3){};\dotnode at (a4){};
\end{tikzpicture}
}
\newcommand{\tzPentagonRadialDivision}{
\begin{tikzpicture}[baseline]
\tikzsetup
\coordinate (a0) at (0:0.7cm){};
\coordinate (a1) at (72:0.7cm){};
\coordinate (a2) at (144:0.7cm){};
\coordinate (a3) at (216:0.7cm){};
\coordinate (a4) at (288:0.7cm){};
\draw[cell,semithick] (a0) --  (a1) -- (a2) -- (a3) -- (a4) -- cycle;
\dotnode at (a0){};\dotnode at (a1){};\dotnode at (a2){};\dotnode at (a3){};\dotnode at (a4){};
\coordinate (O) at (0,0); \dotnode at (O){};
\draw[tikzdefault](O) -- (a0) (O)--(a1) (O) -- (a2) (O) --(a3) (O) -- (a4);
\end{tikzpicture}
}

\newcommand{\tzPentagonElemDivision}{
\begin{tikzpicture}[baseline]
\tikzsetup
\coordinate (a0) at (0:0.7cm){};
\coordinate (a1) at (72:0.7cm){};
\coordinate (a2) at (144:0.7cm){};
\coordinate (a3) at (216:0.7cm){};
\coordinate (a4) at (288:0.7cm){};
\draw[cell,semithick] (a0) --  (a1) -- (a2) -- (a3) -- (a4) -- cycle;
\draw[tikzdefault] (a1)--(a3);
\dotnode at (a0){};\dotnode at (a1){};\dotnode at (a2){};\dotnode at (a3){};\dotnode at (a4){};
\end{tikzpicture}
}

\newcommand{\tzTetrahedron}{
\begin{tikzpicture}[baseline]
\tikzsetup
\coordinate (A) at (0,0){};
\coordinate (B) at (0:1cm){};
\coordinate (C) at (25:1.2cm){};
\coordinate (D) at (0.7cm, 1cm){};
\draw[tikzdefault] (A) --  (B) -- (D) --  cycle (B)--(C)--(D);
\draw[tikzdefault,densely dashed] (A)--(C);
\dotnode at (A){};\dotnode at (B){};\dotnode at (C){};\dotnode at (D){};
\end{tikzpicture}
}
\newcommand{\tzTetrahedronSubdivision}{
\begin{tikzpicture}[baseline]
\tikzsetup
\coordinate (A) at (0,0){};
\coordinate (B) at (0:1cm){};
\coordinate (C) at (25:1.2cm){};
\coordinate (D) at (0.7cm, 1cm){};
\draw[tikzdefault] (A) --  (B) -- (D) -- node (E){} (A) (B)--(C)--(D);
\draw[tikzdefault,densely dashed] (A)--(C);
\dotnode at (A){};\dotnode at (B){};\dotnode at (C){};\dotnode at (D){}; \dotnode at (E){};
\end{tikzpicture}
}
\newcommand{\tzCylinderI}{
\begin{tikzpicture}
\tikzsetup

\begin{scope}[tikzdefault,xscale=0.7]
\coordinate (a1) at (0,0.5);\coordinate (a2) at (7,0.5);
\coordinate (b1) at (0,1.5);\coordinate (b2) at (7,1.5);
\coordinate (c1) at (0,-0.5);\coordinate (c2) at (7,-0.5);
\coordinate (d1) at (0,-1.5);\coordinate (d2) at (7,-1.5);

\path[cell] (a1) arc(-270:89.9:0.5) -- ++(0,1) arc
(89.9:-270:1.5)--cycle;

\fill[yellow!20] (d1) arc (-90:90:1.5) -- (b2) arc (90:-90:1.5)--cycle;

\dotnode  at (a1) {};
\dotnode  at (b1) {};
\dotnode  at (a2) {};
\dotnode  at (b2) {};
\draw[semithick] (0,0) circle (0.5);
 
\draw[semithick] (0,0) circle (1.5);
 
\draw[semithick] (a1)--(a2);
\draw[semithick] (b1)--(b2);
\draw[semithick,densely dashed] (c1)--(c2);
\draw[semithick] (d1)--(d2);
\draw[semithick] (a1)-- (b1);
\draw[semithick] (a2)-- (b2);
\draw[semithick,densely dashed] (7cm, 0cm) circle (0.5);
 
\HalfDashedCircle[xshift=7cm,scale=1.5,semithick]

\end{scope}
\end{tikzpicture}
}

\agtart
\title{On piecewise linear cell decompositions}

\author{Alexander Kirillov, Jr.}
\givenname{Alexander}
\surname{Kirillov}
\address{Department of Mathematics, SUNY at Stony Brook, 
            Stony Brook, NY 11794, USA}
    \email{kirillov@math.sunysb.edu}
    \urladdr{http://www.math.sunysb.edu/\textasciitilde kirillov/}

\newtheorem*{theorem*}{Theorem}
\newtheorem{theorem}{Theorem}[section]
\newtheorem{lemma}[theorem]{Lemma}

\newtheorem{corollary}[theorem]{Corollary}

\theoremstyle{definition}
\newtheorem{definition}[theorem]{Definition}

\newtheorem{example}[theorem]{Example}

\theoremstyle{remark}
\newtheorem{remark}[theorem]{Remark}

\numberwithin{equation}{section}
\newcommand{\firef}[1]{Figure~{\rm\ref{#1}}}
\newcommand{\thref}[1]{Theorem~{\rm\ref{#1}}}

\newcommand{\leref}[1]{Lemma~{\rm\ref{#1}}}
\newcommand{\coref}[1]{Corollary~{\rm\ref{#1}}}

\newcommand{\deref}[1]{Definition~{\rm\ref{#1}}}
\newcommand{\exref}[1]{Example~{\rm\ref{#1}}}

\newcommand{\seref}[1]{Section~{\rm\ref{#1}}}

\newcommand{\st}{\; | \;}                               
\renewcommand{\odot}[1]{\overset{\circ}{#1}}
\newcommand{\mdot}[1]{\overset{\text{\bf\Large .}}{#1}}
\newcommand{\elem}{\underset{\scriptstyle e}\sim}

\newcommand{\del}{\partial}

\newcommand{\isoto}{\xrightarrow{\sim}}       

\newcommand{\ph}{\varphi}

\newcommand{\eps}{\varepsilon}

\DeclareMathOperator{\Int}{Int}

\renewcommand{\cl}{\operatorname{cl}}

\begin{document}

\thanks{This  work was partially suported by NSF grant DMS-0700589 }

\begin{abstract}
In this note, we introduce a class of cell decompositions of PL manifolds
and polyhedra which are more general than triangulations yet not as general
as CW complexes; we propose calling them  PLCW complexes. The main result
is an analog of Alexander's theorem: any two PLCW decompositions of the
same polyhedron can be obtained from each other by a sequence  of certain
``elementary'' moves.

This definition is motivated by the needs of Topological Quantum Field
Theory,  especially extended theories as defined by Lurie. 
\end{abstract}
\maketitle

\section{Introduction}\label{s:intro}
One of the main tools for studying  piecewise-linear manifolds is the
notion  of triangulation, or more generally, cell complexes formed by
convex cells. However, for many purposes this  is too restrictive. For
example, for any  explicit computation of state-sum invariants of
3-manifolds, triangulations turn out  to be a very inefficient tool: the
number of simplices is necessarily quite large, a cylinder over a
triangulated manifold (or, more generally, a product of two triangulated
manifolds) does not have a canonical triangulation, etc. Allowing
arbitrary convex cells helps but does not solve all the problems: for example, a cell
decomposition shown below (which is quite useful for extended
topological field theories and 2-categories, as it illustrates a
2-morphism between two 1-morphisms) can not be realized using only convex
cells. 
$$
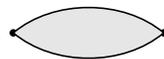
\tzTwoGlobe
$$

In addition, for
many constructions it would be desirable to allow ``singular
triangulations'', where the different faces of the same  cell are allowed
to be glued to each other (for example, this would allow a cell
decomposition of the torus $T^2$ obtained by gluing opposite sides of a
rectangle). On the other hand, CW complexes are too general and  using them
creates other problems: for example, there is no analog of Alexander's
theorem describing simple moves necessary to obtain one CW cell decomposition
from another. 

In this note, motivated by the author's  earlier work with
Balsam~\cite{balsam-kirillov}, we
introduce a new notion of a cell decomposition of a compact polyhedron (in
particular, a PL manifold) which will address many of the problems mentioned
above. We propose calling such cell decompositions PLCW cell
decompositions. We also prove an analog of Alexander's theorem: any two
PLCW decompositions of the same polyhedron can be obtained from each other
by a sequence  of certain ``elementary'' moves; these moves are special
cases of cell  moves introduced by Oeckl~\cite{oeckl}.

\subsection*{Acknowledgments}
The authors would like  to thank
Oleg Viro, Scott Morrison, Benjamin Balsam and Robert Oeckl 
for helpful suggestions and discussions.

This  work was partially suported by NSF grant DMS-0700589.

\section{Basic definitions}\label{s:basic}
In this   section we recall some basic definitions and facts
of PL topology,  following notation and terminology of Rourke and
Sanderson~\ocite{rourke}, where one can also find   the proofs of all
results mentioned here.  

Throughout this paper, the word ``map'' will mean ``piecewise linear map''.
We will write $X\simeq Y$  if there exists a PL homeomorphism $X\to Y$. 

For a subset $X\subset \R^N$, we denote $\Int(X)$ the interior of $X$, by
$\cl(X)$  the closure of $X$ and by $\del X$ the boundary of $X$. We will
also use the following standard notation: 

$B^n=[-1,1]^n\subset \R^n$ --- the $n$-dimensional ball

$S^n=\del B^{n+1}$ --- the $n$-sphere

$\Delta^n\subset \R^{n+1}$--- the $n$--dimensional simplex (note that
$\Delta^n\simeq B^n$)

For any polyhedra $X\subset \R^N$ and a point $a\in \R^N$, we denote $aX$
the cone over $X$.  More generally, given two polyhedra $X,Y\subset \R^N$,
we denote   by $XY$ the join of $X,Y$. When using this notation, we will
always assume that $X,Y$ are independent, i.e. that every $p\in XY$  
can be uniquely written as  $p=ax+by,$ $a,b\in \R$, $a+b=1$.  For two
polyhedra  $X\in \R^n$, $Y\in \R^m$, we denote by $X*Y\subset \R^{n+m+1}$
their external join.

We define a convex $n$-cell $C\subset \R^N$ as a convex compact polyhedron
generating  an affine subspace of dimension $n$; in such a situation, we will
also write $\dim C=n$. In Rourke and Sanderson~\cite{rourke}, these are
called just cells; we prefer a more specific name to avoid confusion with
other types of cells to be introduced later. 

For any such cell we can define the set $F(C)$ of faces of $C$ (of arbitrary codimension); 
each face $F$ is itself a convex cell. We will write $F<C$ if $F, C$ are convex cells
and $F$ is a face of $C$.

Recall that  each convex  cell $C$ is homeomorphic to
a ball: $C=\ph(B^n)$ for some homeomorphism $\ph$. As usual, we denote 
\begin{align*}
 \odot{C}&=\Int(C)=\ph(\Int(B^n))\\
 \mdot{C}&=\del C=\ph(\del B^n)
 \end{align*}
if $\dim C>0$. If $\dim C=0$, i.e. $C$ is a point, then we let
$\odot{C}=C$, $\mdot{C}=\varnothing$. 

Following  \ocite{rourke}, we define  a cell complex $K$  as a finite
collection of convex cells in $\R^N$ such  that the following conditions are 
satisfied:
\begin{enumerate}
\item If $A\in K$ and $B<A$, then $B\in K$
\item If $A,B\in K$, and $F=A\cap B\ne \varnothing$, then $F<A$, $F<B$.
\end{enumerate}

We define the support $|K|=\cup_{C\in K} C$; it is a compact polyhedron in $\R^N$. 
Conversely, given a compact polyhedron $X$, a {\em cell decomposition} of $X$ is a 
complex $K$ such that $|K|=X$; it is known that such a decomposition always exists. 
We will denote by $\dim K$ the dimension of $K$ and by $K^n$ the $n$-skeleton of $K$. 
Given a complex $K$ and a cell $C$, we will denote $K+C$ the complex obtained by 
adding to $K$ the cell $C$ assuming that it does form a complex.

In particular, given a convex cell $C$, the set $F(C)$ of faces of $C$ is a
cell  complex, with $|F(C)|=\mdot{C}$; by adding to it $C$ itself, we get a
cell decomposition  of $C$.

\section{Generalized cells}\label{s:gen_cell}

Let $C$ be a convex cell in $\R^N$. 
\begin{definition}
 A map $f\colon C\to \R^m$ is called {\em regular} if the restriction
$f|_{\odot C}$ is injective.  
\end{definition}

\begin{lemma}\label{l:regular1}
If $C$ is a convex cell and $f\colon C\to \R^m$ is regular, then $C$ admits
a cell decomposition $K$  such that for any cell $K_i\in K$, the
restriction $f|_{K_i}$ is injective.
\end{lemma}
\begin{proof}
By standard results of PL topology, $C$ admits a cell decomposition 
such that $f|_{K_i}$ is linear, and a linear map which  is injective on an
open set is injective. 
\end{proof}

We can now define the generalization of the notion of a convex cell. 

\begin{definition}\label{d:gen_cell}
A   {\em generalized $n$-cell} is a subset $C\subset \R^N$ together with 
decomposition $C=\odot{C}\sqcup \mdot{C}$ such that $\odot{C}=\ph(\Int B^n)$, 
$\mdot{C}=\ph(\del B^n)$ (and thus $C=\ph(B^n)$) for some regular map $\ph\colon B^n\to \R^N$. 

In such a situation, the map $\ph$ is called a {\em characteristic map}.
\end{definition}
Note that the definition implies that $C=\cl(\odot{C})$, so $C$ 
is completely determined by $\odot{C}$. It is also clear from
\leref{l:regular1} that any generalized cell is a compact polyhedron. 

Clearly any  convex cell is automatically a generalized cell. 
Other examples of generalized cells are shown in \firef{f:gcc1} below. 

Note that characteristic map $\ph$ in the definition of generalized 
cell is not unique. However, as the following theorem shows, it is unique 
up to a PL homeomorphism of the ball.
\begin{theorem}\label{t:char_map}
Let $C\subset \R^N$ be a generalized cell and $\ph_1,\ph_2\colon B^n\to C$ 
be two characteristic maps. Then there exists a unique homeomorphism 
$\psi\colon B^n\to B^n$ such that $\ph_1=\ph_2\circ \psi$. 
\end{theorem}
\begin{proof}
Since restriction of $\ph_i$ to $\Int(B^n)$ is injective,  the composition
$\odot\psi= \ph_2^{-1}\ph_1$  is well defined as a map $\Int(B^n)\to
\Int(B^n)$. To show that it can be extended to the boundary, note that it
follows from \leref{l:regular1} that one can find a cell
decomposition $K$ of $B_n$ such that
$\odot\psi|_{K_i}$ is linear for every $n$-cell $K_i\in K$. This
immediately implies that $\odot\psi$ can be extended to a homeomorphism 
$\psi\colon B^n\to B^n$.

\end{proof}
It is easy to show that cone and join of generalized cells is again a 
generalized cell. Namely, if $C=\ph(B^n)$ is a generalized cell, and $aC$ 
is the cone of $C$, then the map $\ph$ can be in an obvious way lifted 
to a map $\{pt\}*B^n\simeq B^{n+1}\to aC$, which is easily seen to be 
regular. Thus, $aC$  is a generalized cell.  In the similar way, using 
homeomorphism $B^m*B^n\simeq B^{m+n+1}$, one shows that if $C_1, C_2$ 
are generalized cells that are independent, then the join $C_1C_2$ is 
also a generalized cell.

\section{Generalized cell complexes}\label{s:gen_complex}

{\bf From now on, unless noted otherwise, the word ``cell'' stands for a 
generalized cell.}

\begin{definition}\label{d:gen_complex}
A generalized cell complex (g.c.c.) is a finite collection $K$ of
generalized cells in $\R^N$  such that 
\begin{enumerate}
\item for any distinct $A, B$ in $K$, we have 
$$
\odot{A}\cap \odot{B}=\varnothing
$$
\item For any cell $C\in K$, $\mdot{C}$ is a union of cells. 
\end{enumerate}

Support $|K|\subset \R^N$ of a generalized cell   complex $K$ is defined by 
$$
|K|=\bigcup_{C\in K}C
$$

A generalized cell  decomposition of a compact polyhedron
$P\subset \R^N$ is a generalized complex $K$ such that $|K|=P$. 
\end{definition}

We define the dimension $\dim K$ of a generalized cell complex
and the $n$-skeleton $K^n$ in the usual way. Also, if $A,B\in K$ are 
cells such that $A\subset \mdot{B}$, we will say that $A$ is a face of $B$ 
and write $A<B$; clearly this is only possible if $\dim A<\dim B$. 

If $K,L$ are g.c.c., we denote by $K+L$ the complex obtained by 
taking all cells of $K$ and all cells of $L$, assuming that the 
result is again a g.c.c.

\begin{example}\label{x:gcc}
\begin{enumerate}
\item Any cell complex is automatically a g.c.c. 
\item A 0-dimensional g.c.c. is the same as finite collection of points. A
1-dimensional  g.c.c. is the same as a finite collection of points
(vertices) and  non-intersecting arcs (1-cells) with endpoints at these
vertices.  Note that loops are allowed. 
\item \firef{f:gcc1} shows some examples of 2-dimensional g.c.c. 
\begin{figure}[ht]
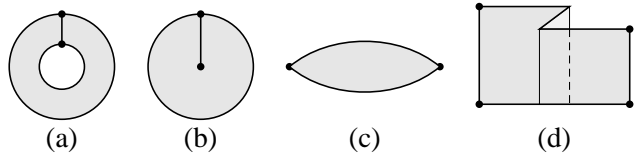


\begin{tabular}{cccc}
\tzAnnulusWithRadius&
\tzCircleWithRadius&
\tzTwoGlobe&
\tzZigZagComplex\\
(a)&(b)&(c)&(d)
\end{tabular}
\caption{Examples of 2-dimensional generalized cell complexes.
The last one can be visualized as a sheet of paper with a fold, 
with the lower edge glued back to itself. Note that it only has four 1-cells: 
the lines showing where the paper was folded  are not 1-cells.}
\label{f:gcc1}
\end{figure}

\item \firef{f:gcc2} shows a generalized cell decomposition of $S^1\times I\times I$ 
consisting of a single 3-cell, five 2-cells, eight 1-cells and 4 vertices.
\begin{figure}[ht]
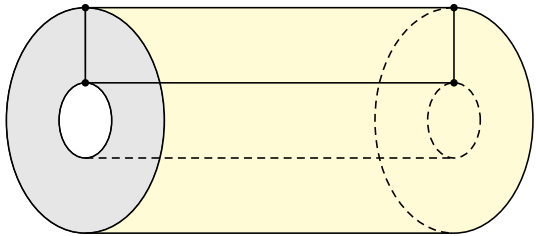

\tzCylinderI
\caption{A generalized cell decomposition of $S^1\times I\times
I$}\label{f:gcc2} 
\end{figure}
 
\end{enumerate}
\end{example}

\begin{definition}\label{d:regular_cellular}
 Let $K,L$ be g.c.c. A {\em regular cellular map} $f\colon L\to K$ is a map
$f\colon |L|\to |K|$ such that for every cell $C\in L$, $C=\ph(B^n)$,
there exists a cell $C'\in K$ such that $C'=f(C)$ and moreover, $f\circ
\ph\colon B^n\to C' $ is a characteristic map for $C'$. 
 \end{definition}
In other words, such a map is allowed to identify different 
cells of $L$ but is injective on the interior of each cell. 

An example of a regular cellular map is shown in
\firef{f:regular_cellular}.
\begin{figure}[ht]
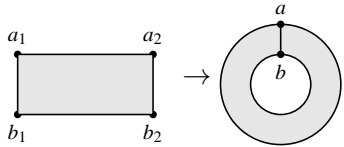

$\tzRectangle\to \tzAnnulusWithRadiusLabeled$
\caption{An example of a regular cellular map.
It identifies edges $a_1b_1$ and $a_2b_2$,
sending each of them to edge $ab$.}\label{f:regular_cellular}
\end{figure}
\section{PLCW complexes}\label{s:PLCW}
In this section, we give the central definition of the paper.

\begin{definition}\label{d:PLCW}
 A generalized cell complex  (respectively, a generalized cell
decomposition) $K$ will be called a {\em PLCW complex} (respectively, PLCW
decomposition) if $\dim K=0$, or $\dim K=n>0$ and the following
conditions holds:
\begin{enumerate} 
  \item $K^{n-1}$ is a PLCW complex
  \item For any $n$-cell $A\in K$, $A=\ph(B^n)$, there exists a PLCW
  decomposition $L$ of $\del B^n$ such that the restriction 
  $\ph|_{\del B^n}\colon L\to K^{n-1}$   is a regular  cellular map. 
  (It follows from \thref{t:char_map} that this condition is independent 
  of the choice of characteristic map $\ph$.)
\end{enumerate}
\end{definition}

In other words, a PLCW is obtained by successively attaching balls, and the attaching map
should be a regular cellular map for some PLCW decomposition of the boundary sphere. 

Note that this definition is inductive: definition of an  $n$-dimensional PLCW
complex uses definition of an $(n-1)$ dimensional PLCW complex. 

\begin{example}\label{x:PLCW}
 Among examples in \exref{x:gcc}, example 2(d) is not a PLCW complex. All
other are PLCW. 
\end{example}
It is easy to show that for an $n$-cell $A\in K$ and fixed choice of characteristic 
map $\ph\colon B^n\to A$, the generalized cell decomposition $L$ of $\del B^n$ used in 
\deref{d:PLCW} is unique. Indeed, the cells of $L$ are closures of connected components
of $\ph^{-1}(\odot{K_i})$, $K_i\in K^{n-1}$. We will call such an $L$ the 
pullback of  $K$ under the map $\ph$ and denote it by 
\begin{equation}\label{e:pullback}
L=\ph^{-1}(K).
\end{equation}

The following properties of PLCW complexes are immediate from the definition.

\begin{enumerate}
\item $|K|=\sqcup_{C\in K}\odot{C}$
\item If $A,B\in K$ are two cells, then $A\cap B$ is a union of cells of $K$. 
\item For any $n$-cell $C\in K$, $\mdot{C}$ is a union of $(n-1)$-cells of $K$.
\item Every PLCW complex is automatically a CW complex. 
\end{enumerate}

Note that not every CW complex is a PLCW complex, even if its cells are polyhedra.
For example, property (3) could fail for more  general CW complexes.

The following two lemmas, proof of which is straightforward and left to the reader, 
show that product and join of PLCW complexes is a PLCW complex.

\begin{lemma}\label{l:PLCWproduct}
Let $K,L$ be PLCW complexes in $\R^M$, $\R^N$ respectively. Define the complex 
$$
K\times L=\sum K_i\times L_j\subset \R^M\times \R^N.
$$
Then $K\times L$ is a PLCW complex with support $|K|\times |L|$. 
\end{lemma}

\begin{lemma}\label{l:PLCWjoin}
Let $K,L$ be PLCW complexes in $\R^N$ such that $|K|$, $|L|$ are independent: every point 
$p\in |K| |L|$ can be uniquely written in the form $p=ax+by$, $x\in |K|$, $y\in |L|$, 
$a,b\ge 0$, $a+b=1$. Define the join  of them by
$$
KL=K+L+\sum K_i L_j, 
\qquad K_i\in K, \quad L_j
\in L
$$
Then $KL$ is a PLCW complex with support $|K| |L|$.
\end{lemma}
The proof is straightforward and left to the reader. 

Note that in the case $K=\{a\}$ --- a point, we see that the cone 
$$
aL=a+L+\sum aL_i,\quad L_i\in L
$$
of a PLCW complex is a PLCW complex. 
\section{Subdivisions}\label{s:subdivision}
\begin{definition}\label{d:subdivision}
Let $K, L$ be PLCW complexes. We say that $L$ is a subdivision of $K$ (notation: $L\lhd K$) 
if $|K|=|L|$ and for any cell $C\in K$, we have $\odot{C}=\cup \odot{L_i}$ for some 
collection of cells $L_i\in L$. 
\end{definition}
Note that this implies that any cell $L_i\in L$ is a subset of one of the cells of $K$ 
(which is the usual definition of subdivision of cell complexes). Moreover, it is easy to 
see that if $K,L$ are cell complexes, then this definition is actually equivalent 
to the usual definition of subdivision.

There is a special kind of subdivisions we will be interesed in.

\begin{definition}\label{d:radial_subdivision}
Let $K$ be a PLCW complex, $C=\ph(B^n)$ an $n$-cell, $n>0$ and $L=\ph^{-1}(K)$ 
the pullback cell decomposition on $\del B^n$ (see \eqref{e:pullback}). 
We define the {\em radial subdivision}
of $K$ to be the subdivision obtained by replacing the cell $C$ by the 
cone PLCW complex $\ph(O)+\ph(OL_1)+\dots+\ph(OL_k)$, where 
$L=\{L_1,\dots, L_k\}$ and $O\in \Int(B^n)$ is the origin. 
(Recall that a cone of a PLCW complex is PLCW complex, see \leref{l:PLCWjoin}.)  
\end{definition}

\firef{f:radial_subdivision} shows examples of radial subdivisions.
\begin{figure}[ht]
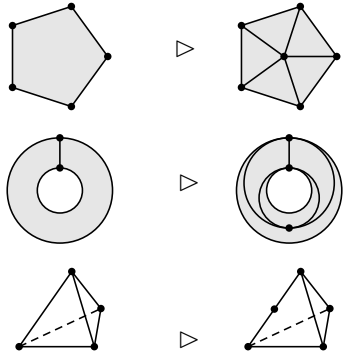

\begin{tabular}{ccc}
\tzPentagon &\quad $\rhd$ \quad &\tzPentagonRadialDivision\\ \noalign{\smallskip\smallskip\smallskip}
\tzAnnulusWithRadius & \quad$\rhd$\quad & \tzAnnulusWithRadiusSubdivision\\ \noalign{\smallskip\smallskip\smallskip}
\tzTetrahedron &\quad$\rhd$ & \tzTetrahedronSubdivision
\end{tabular}
\caption{Examples of radial  subdivisions.
Note that in the last example, we are subdividing a 1-cell.
}\label{f:radial_subdivision}
\end{figure}

Note that this is very closely related to the usual notion of stellar subdivision 
for simplicial complexes but it is not identical to it. Namely, for radial 
subdivision we are subdividing just one cell $C$ without changing the 
higher dimensional cells  adjacent to $C$ 
(see the last example in \firef{f:radial_subdivision}). Comparing it with the definition of the 
stellar subdivision, we see that if $K$ is a simplicial complex, $C\in K$ 
--- an $n$-cell, and $L$--- the stellar subdivision of $K$ obtained by starring 
at $a\in \odot{C}$, then $L$ can also be obtained by 
\begin{enumerate}
\item Replacing $C$ by the radial subdivision $R$
\item Replacing every cell $A=CB$ in the star of $C$ by the complex $R_iB$, $R_i\in R$. 
\end{enumerate}

\begin{theorem}\label{t:triangulation}
Any PLCW complex $K$ has a subdivision $T\lhd K$ which is a triangulation; 
moreover, $T$ can be obtained from $K$ by a sequence of radial subdivisions.
\end{theorem}
\begin{proof}
Let $K'$ be obtained from $K$ by radially subdividing of each cell of $K$ of positive dimension
 in order of increasing dimension. Then it is easy to see that $K'$ has the following property:

\begin{equation}
\text{For any $C\in K'$, the characteristic map $\ph\colon B^n\to C$ is injective}
\end{equation}

Now, let $T$ be obtained by again doing the radial subdivision of each cell of $K'$ in order 
of increasing dimension. It is easy to see that  $T$ is a triangulation: this follows 
by induction from the fact that given a triangulation $\mdot{T}$ of $S^{n-1}$, the radial subdivision
$a\mdot{T}$ of $B^n$ is a triangulation (which in turn follows from the fact that the cone 
over a simplex is a simplex).
\end{proof}
\section{Elementary subdivisions}\label{e:elem_subdivisions}

The other type of subdivision will be called {\em elementary subdivision}. 
Informally, these are obtained by dividing an $n$-cell into two $n$-cells 
separated by an $(n-1)$-cell. To give a more formal definition,  we need some notation.

Let $H_0\subset \R^n$ be hyperplane defined by equation $x_n=0$. It divides $\R^n$ into two subspaces: 
\begin{equation}\label{e:halfspaces}
\begin{aligned}
H_+&=\{(x_1,\dots, x_n)\in \R^n\st x_n\ge 0\}\\
H_-&=\{(x_1,\dots, x_n)\in \R^n\st x_n\le 0\}
\end{aligned}
\end{equation}

For the $n$-ball $B^n\subset \R^n$, define upper and lower halfballs 
\begin{equation}
B^n_+=B^n\cap H_+, \qquad B^n_-=B^n\cap H_-
\end{equation}
We also define the middle disk and the equator by 
\begin{equation}
B^n_0=B^n\cap H_0\simeq B^{n-1}, \qquad E=S^{n-1}\cap H_0\simeq S^{n-2}
\end{equation}

\begin{lemma}\label{l:elem_subdivision}
Let $K$ be a PLCW and $C=\ph(B^n)$ --- an $n$-cell. Assume that the pullback 
decomposition  $L=\ph^{-1}(K)$ of $\del B^n$ is such that the equator $E\subset \del B^n$ 
is a union of cells of $L$. Let $K'$ be the g.c.c. obtained by replacing $C$ by the 
collection of cells $C_+=\ph(B^n_+)$, $C_-=\ph(B^n_-)$, $C_0=\ph(B^n_0)$. 

Then $K'$ is a PLCW complex; moreover, $K'$ is a subdivision of $K$.
\end{lemma}

\begin{definition}\label{d:elem_subdivision}
Let $K,K'$ be as in \leref{l:elem_subdivision}. Then we say that $K'$ is obtained from $K$
by an elementary subdivision of cell $C$; we will also say that $K$ is obtained from 
$K'$ by erasing cell $C_0$. 

We will write $K\elem L$ if $K$ can be obtained from $L$ by a finite sequence of elementary 
subdivisions and their inverses.
\end{definition}

Note that elementary subdivisions are essentially the same as 
$(n,n)$ moves in  intruduced by Oeckl in \cite{oeckl1} and further studied Oeckl's book \cite{oeckl2}; in Oeckl's work, these moves are special case of a more general moves called $(n,k)$ moves.

An example of elementary subdivision is shown in \firef{f:elem_subdivision}.
\begin{figure}[ht]
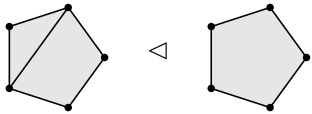

\tzPentagonElemDivision\quad$\lhd$\quad  \tzPentagon
\caption{An elementary subdivision}\label{f:elem_subdivision}
\end{figure}

\begin{remark}
Not every subdivision can be obtained by a sequence of elementary subdivisions. For example, 
the subdivision shown in \firef{f:nonelem_subdivision} can not be obtained by a sequence of 
elementary subdivisions. 
\begin{figure}[ht]
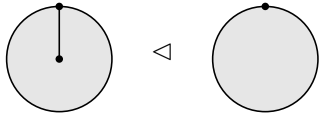

\tzCircleWithRadius\quad $\lhd$\quad \tzCircle
\caption{A  non-elementary subdivision}\label{f:nonelem_subdivision}
\end{figure}
However, it can be obtained by a sequence of elementary subdivisions 
{\bf and their inverses} as shown in \firef{f:nonelem_subdivision2}.

\begin{figure}[ht]
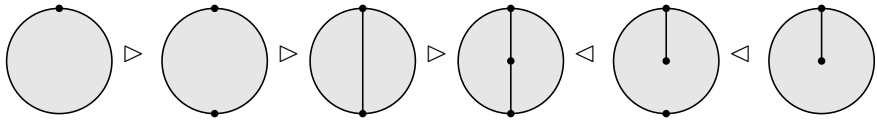

	\tzCircle $\rhd$ \tzCircleII$\rhd$\tzCircleIII$\rhd$\tzCircleIV
	$\lhd$ \tzCircleV$\lhd$ \tzCircleWithRadius
    \caption{Obtaining a non-elementary subdivision by a sequence of elementary 
			subdivisions  and their inverses}\label{f:nonelem_subdivision2}
\end{figure}

\end{remark}

\begin{theorem}\label{t:elem_join}
	If  $M=KL$ is a join of two PLCW complexes and $K'\lhd K$ --- an elementary subdivision of $K$, 
	then $M'=K'L$ be obtained from $M$ by a sequence of elementary subdivisions.
\end{theorem}
\begin{proof}
	If $C\in K$ is an $n$-cell and $C=C_++C_-+C_0$ its elementary subdivision as in 
	\leref{l:elem_subdivision}, and $D$ is a cell in $L$, then $CD=C_+D + C_-D + C_0D$ 
	is an elementary subdivision of $CD$, which follows from existence  
	of  a homeomorphism $\psi\colon B^n*B^m\isoto B^{m+n+1}$ such that 
	$\psi(B_0^n*B^m)=B_0^{m+n+1}$, $\psi(B_\pm^n*B^m)=B_\pm^{m+n+1}$. Repeating it for 
	every cell $D\in L$ in order of increasing dimension, we see that $K'L$ can be obtained 
	from $KL$ by a sequence of elementary subdivisions.
\end{proof}
\begin{corollary}\label{c:elem_join}
	If $K\elem K'$, then $KL\elem K'L$. 
\end{corollary}

\section{Main theorem}\label{s:main}
In this section, we formulate and prove the main theorem of this paper. Recall the notation 
$K\elem L$ from \deref{d:elem_subdivision}.

\begin{theorem}\label{t:main1}
Let $K, K'$ be two PLCW decompositions of a compact polyhedron $X$. Then $K\elem K'$.
\end{theorem}
\begin{proof}

This proves (for PLCW decompositions) the conjecture of Oeckl \cite{oeckl2}:
that any any cell decompositions can be obtained form each other by a
sequence of $(,n,k)$ moves; in fact, it proves a stronger result, that
$(n,n)$ moves are already enough. 

We proceed by induction in $n=\dim X$. If $n=0$, there is nothing to
prove. So from now on, we assume that $n>0$ and that the theorem is already
proved for all polyhedra of dimension less  than $n$. 

{\bf Step 1.}
Let  $X=B^n$ be an $n$-ball,  $\mdot{K}$ -- a PLCW decomposition of $S^{n-1}=\del B^n$, and 
$R=a\mdot{K}$ --- the corresponding radial cell decomposition of $X$, $a\in \Int(B^n)$.  Then 
$R\elem B^n+\mdot{K}$.

Indeed, let $L$ be a PLCW decomposition of $S^{n-1}$ consisting of the  upper  and lower 
hemispheres $S^{n-1}_\pm\simeq B^{n-1}$ and some PLCW decomposition $L_0$ of the equator $E$. 
By induction assumption, $\mdot{K}\elem L$; by \coref{c:elem_join}, this implies 
$$
a\mdot{K}\elem aL= B_+^n+B^n_- +S^{n-1}_+ +S^{n-1}_- + aL_0.
$$ 
By using the induction assumption again, $aL_0\elem B^n_0 +L_0$, so 
$$
a\mdot{K}\elem  B_+^n+B^n_- +B^n_0 +S^{n-1}_+ +S^{n-1}_- +L_0\elem B^n+S^{n-1}_+ + S^{n-1}_- + L_0\elem B^n+\mdot{K}
$$

{\bf Step 2.}
If $K'$ is obtained from $K$ by a sequence of radial subdivisions, then $K'\elem K$.

This follows from the previous step and definition. 

{\bf Step 3.}
For any PLCW decomposition $K$, there is a triangulation $T$ such that $K\elem T$.

Indeed, it follows from the previous step and \thref{t:triangulation}.

{\bf Step 4.} If $T, T'$ are triangulations of $X$, then $T\elem T'$.

By Alexander's theorem, $T$ can be obtained from $T'$ by a sequence of stellar moves, so it 
suffices to prove the theorem in the case when $T'$ is obtained from $T$ by starring at 
point $a\in \Int(C)$ for some simplex $C\in T$.  
By the discussion in \seref{s:subdivision}, we can also describe $T'$ by replacing $C$ by the 
radial subdivision $C'$ of $C$ and replacing every simplex $A=CB$ in the star of $C$ by $C'B$. 
By step 2 and \coref{c:elem_join}, this implies that $T'\elem K$. 

Combining steps 3 and 4 above, we arrive at the statement of the theorem. 
\end{proof}

\section{Orientations}\label{s:orientation}
Recall that the group of homeomorphisms of $B^n$ has a homomorphism 
to $\Z_2$, called {\em orientation}. Using this, we can define the 
notion of orientation of a cell.

\begin{definition}\label{d:orientation}
Let $C\subset \R^N$ be a generalized $n$-cell. An {\em orientation} of 
$C$ is an equivalence class of characteristic maps $B^n\to C$, where two 
characteristic maps $\ph_1,\ph_2\colon B^n\to C$ are equivalent if 
$\psi=\ph_2\ph_1^{-1}\colon B^n\to B^n$ is orientation-preserving. 

An oriented cell $\mathbf C=(C, [\ph])$ is a pair consisting of a cell $C$ 
and an orientation $[\ph]$. 
\end{definition}

Note that any convex $n$-cell $C\subset \R^n$ has a canonical orientation. Moreover,
if $C\subset \R^n$ is is a convex $n$-cell, and $D\subset \del C$ is a generalized  
$(n-1)$-cell, then $D$ has a canonical orientation  defined by the usual condition:
\begin{equation}\label{e:orientation_boundary}
\eps(\mathbf{C}, \mathbf{D})=1
\end{equation}
where $\eps(\mathbf{C}, \mathbf{D})$ is the incidence number, defined in
the same way as for CW cells (see, e.g., \cite{rourke}*{Appendix A.7}).

Thus, if $C$ is is a convex $n$-cell in $\R^n$, and $L$ -- a PLCW decomposition 
of $\del C$, then each of $(n-1)$-cells $L_i\in L$ has a canonical orientation.

The following definition generalizes this to an arbitrary oriented cell.

\begin{definition}\label{d:delC}
Let $K$ be a PLCW complex, and $\mathbf{C}=(C, [\ph])$ --- an oriented cell. 
Let $L=\ph^{-1}(K)$ be the pullback decomposition of $\del B^n$. 
We define the boundary $\del \mathbf{C}$ as a multiset  (set with multiplicities) 
of oriented $(n-1)$-cells
$$
\del \mathbf{C}=\bigcup (\ph(L_i), [\ph\circ \ph_i])
$$
where the union is over all $(n-1)$-cells $L_i\in L$, each taken 
with the natural orientation $[\ph_i]$ defined by \eqref{e:orientation_boundary}. 
\end{definition}
It is easy to see, using \thref{t:char_map}, that this definition does not depend 
on the choice of characteristic map $\ph$ in the equivalence class.

Note that by definition of a PLCW,
 for each $L_i\in L$, $\ph(L_i)$ is an $(n-1)$-cell of $K$; however, the same $(n-1)$-cell $D\in K$ 
can appear in $\del C$ more than once, and possibly with different
orientations.  Note also that passing from the multisets to the abelian
group generated by oriented cells, we get the usual definition of the
boundary operator in the chain complex of a CW complex. However, for
applications to topological field theory, the  definition of the boundary as a
multiset is much more useful.

\begin{example}
Let $C$ be the $2$-cell shown below. Then $\del C=\{a,\bar a, b\}$, where $\bar a$ 
denotes $a$ with opposite orientation. 
$$\tzCircleWithRadiusOrient
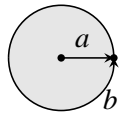
$$
\end{example}

The proof of the following lemma is left to the reader as an exercise.
\begin{lemma}
 Let $X$ be an oriented PL manifold with boundary  and $K$ --- a PLCW
decomposition of  $X$. Then 
$$
\cup_{\mathbf{C}} \del \mathbf{C}=
\Bigl(\cup_{D} \mathbf{D}\Bigr) \cup
\Bigl(\cup_{F} \mathbf{F}\cup \mathbf{\bar F} \Bigr)
$$
where 
\begin{itemize}
\item $C$ runs over all $n$-cells of $K$, each taken with orientation 
	induced by orientation of $X$
\item $D$ runs over all  $(n-1)$ cells such that $D\subset \del X$, 
	each taken with orientation induced by orientation of $\del X$  
\item $F$ runs over all \textup{(}unoriented\textup{)} $(n-1)$-cells such
	that $\odot F\subset \Int(X)$; $\mathbf{F}$ and $\mathbf{\bar F}$ are the
	two possible orientations of $F$. 
\end{itemize}
\end{lemma}



\end{document}